\documentclass[11pt]{article}
\usepackage{amsmath, amsfonts, amsthm, amssymb}  
\usepackage{fullpage}

\usepackage[x11names, rgb]{xcolor}
\usepackage{graphicx}
\usepackage{tikz}
\usetikzlibrary{decorations,arrows,shapes}

\usepackage{etoolbox}
\usepackage{enumerate}
\usepackage{listings}
\usepackage{booktabs}
\usepackage{apacite}
\usepackage{comment}
\usepackage[affil-it]{authblk}

\def\indented#1{\list{}{}\item[]}
\let\indented=\endlist

\newtheorem{theorem}{Theorem}[section]
\newtheorem{lemma}[theorem]{Lemma}
\newtheorem{proposition}[theorem]{Proposition}

\title{Number of regions created by random chords in the circle}
\author{Shi Feng }
\affil{University of Washington, Seattle}
\date{June 2021}

\begin{document}

\maketitle
\begin{abstract}
In this paper we discuss the number of regions in a unit circle after drawing $n$ i.i.d. random chords in the circle according to a particular family of distribution. We find that as $n$ goes to infinity, the distribution of the number of regions, properly shifted and scaled, converges to the standard normal distribution and the error can be bounded by Stein's method for proving Central Limit Theorem.
\end{abstract}

\section{Introduction}
\subsection{Literature review}
Geometric probability has an old history starting $17^{th}$ century \cite{kalousova2011origins}. However, we usually consider Buffon to be the starter for geometric probability in $18^{th}$ century, who raised the ``Buffon's Needle Problem " \cite{buffon1777essai}. We focus on one of the branches in geometric probability called the random chord. One of the most famous problem about random chords is the Bertrand paradox \cite{bertrand1889calcul}. It asks what is the probability for a random chord in a unit circle to have a length larger than $\sqrt{3}$. Then in 1964, David and Fix started to investigate random chords rigorously to study random mechanism of chromosomes in cell \cite{david1964intersections}, where five randomization models are described. The result is repeated in the book $\textit{Geometric Probability}$ \cite{solomon1978geometric}, from where we will borrow the definitions for our problem. 

Consider the unit circle. A chord can be identified by two parameters, the distance from the origin and the direction of the chord with respect to the origin. To create a random chord, we use the half the angle subtended by the chord ($\theta$, as shown in Figure 1) to measure the distance from the origin and use the position of the left endpoint from the perspective of the origin (shown as the point $C$ in Figure 1) to measure direction. We assume $\theta$ is a continuous random variable with a probability density function $f$ (we call it distance distribution) in domain $[0,\frac{\pi}{2}]$, and the left endpoint is uniformly distributed on the circumference of the circle, and they are independent. With these two parameters, we can define every chord in the circle.

\begin{figure}[h]
  \centering
  \includegraphics[scale=3]{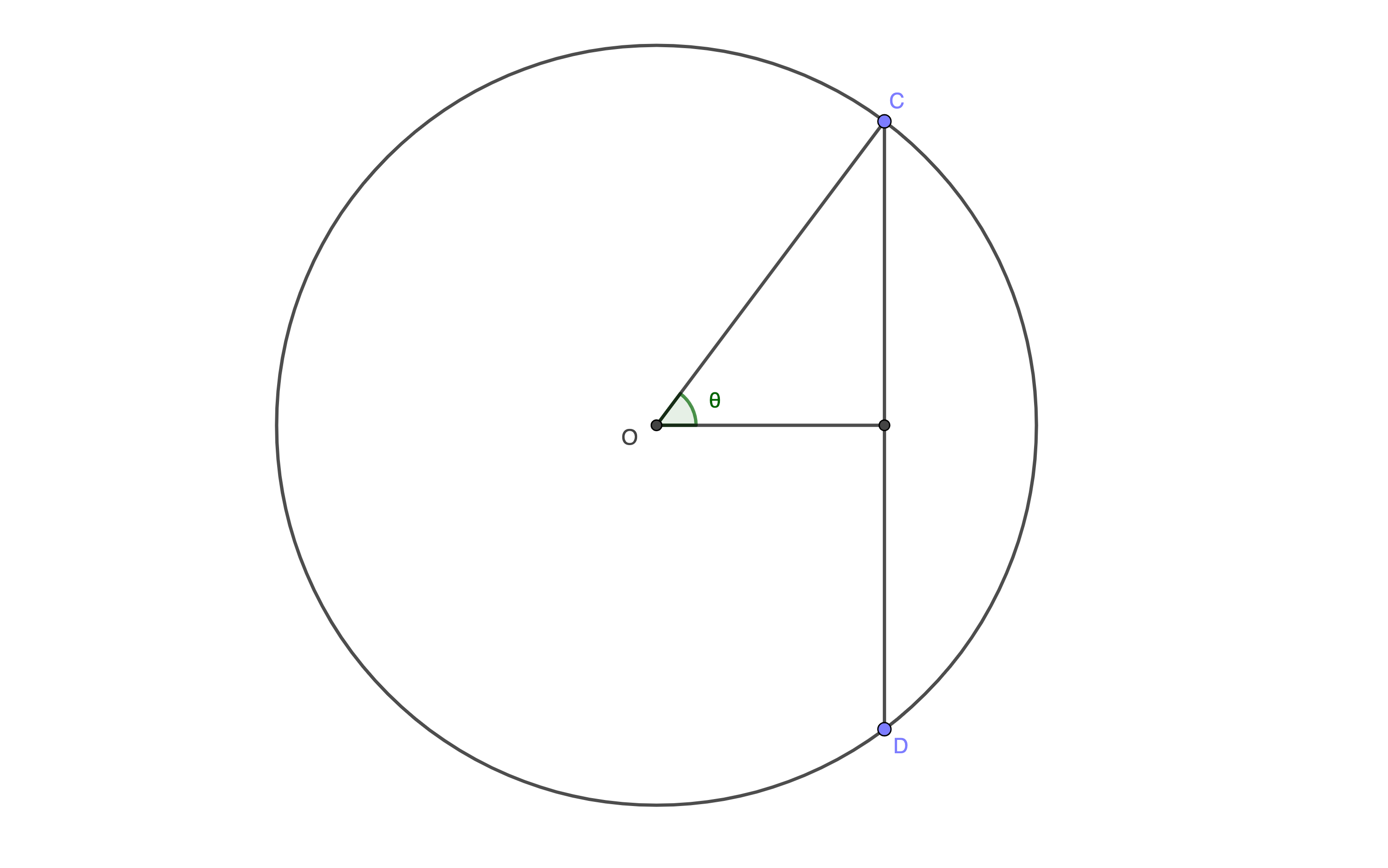}
  \caption{$\overline{CD}$ is the random chord and $\theta$ is the angle that represents the chord.}
\end{figure}

This paper focus on the number of regions in the circle created by those random chords. After drawing $n$ chords, the unit disc is partitioned into disjoint regions through which no chord passes. For example, if no chord is drawn, there is only one region, which is the area enclosed by the circle. After drawing one random chord, the circle is cut into two parts and there are two regions. If two random chords are drawn, there will be two cases. If the two chords have an intersection, there will be 4 regions. On the other hand, if the two chords have no intersection, there will only be 3 regions. Our question is the distribution of the number of regions after drawing $n$ random chords when $n$ is large. For example, in Figure 2, 4 chords are drawn and there are 8 regions in the circle.

\begin{figure}[h]
  \centering
  \includegraphics[scale=3]{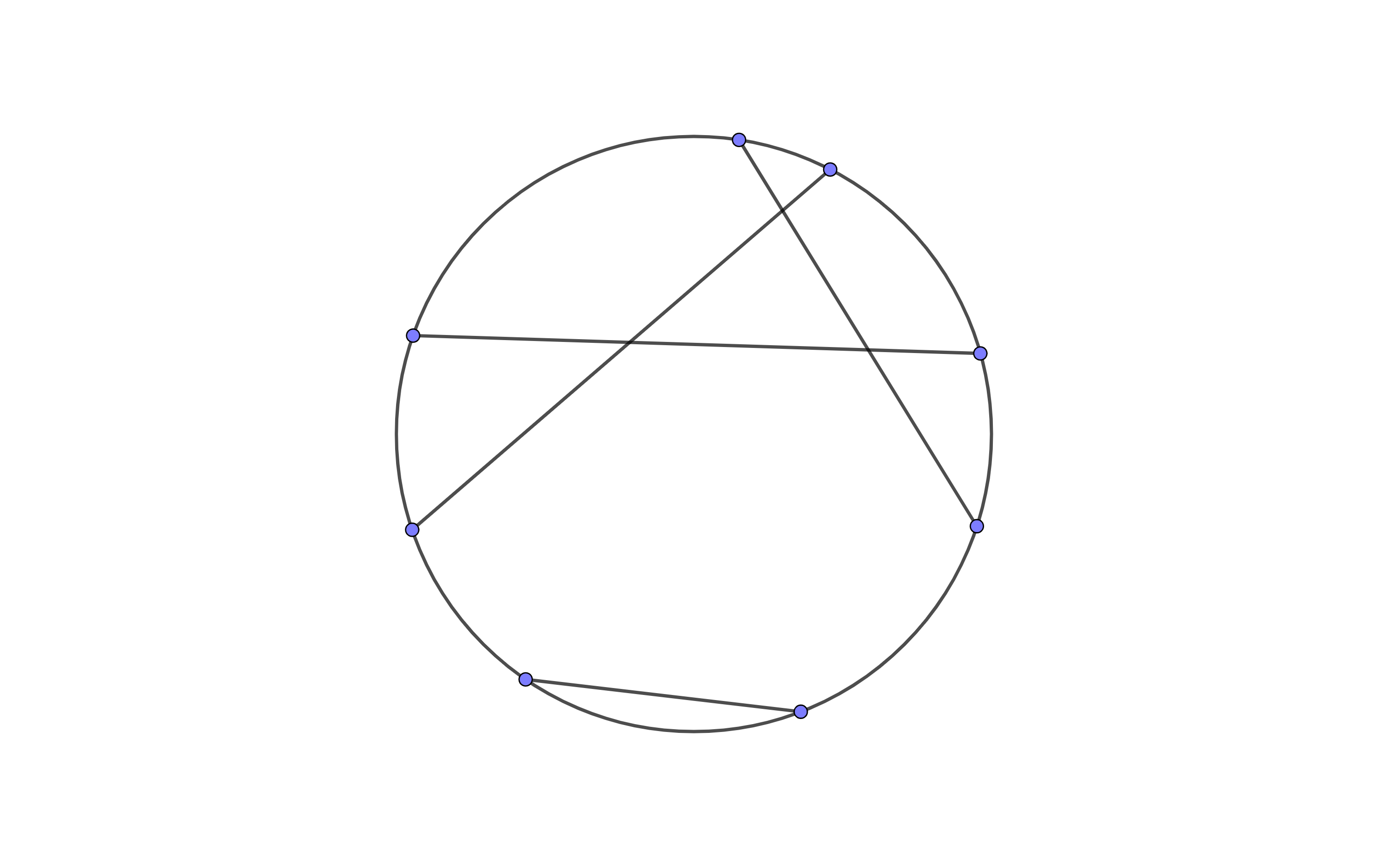}
  \caption{4 chords in the circle create 8 regions}
\end{figure}

Closely related to our paper is the Poisson line process (PLP), which is derived from Poisson point process (PPP). A planar (2D) line can be identified by two parameter \cite{dhillon2020poisson} ``(i) the perpendicular distance $y$ of the line from the origin, and (ii) the angle $x$ subtended by the perpendicular dropped onto the line from the origin with respect to the positive x-axis in counterclockwise direction." A planar PLP can be mapped from a PPP in a 2D plane where $x$ is in range $[0,2\pi)]$ and $y$ is in range $[0,\infty)$. The investigation of PLP starts since 1940s by S. Goudsmit \cite{goudsmit1945random}. The following research includes topics about PLP hitting a convex region and polygons created by PLP by Miller \cite{miles1964random}. Now, Poisson line processes have been applied in various fields, ``including material science,
geology, image processing, transportation, localization, and wired and wireless communications"\cite{dhillon2020poisson}.

\subsection{Computer simulation}
Let's do a computer simulation to see what does the distribution of the number of regions look like when $n$ i.i.d. chords are drawn from the above model for large $n$. Consider the probability density function
\[ f(\theta) =
  \begin{cases}
    \sin(\theta),       & \quad \text{if $0\leq\theta\leq\frac{\pi}{2}$}\\
    0,  & \quad \text{else}.
  \end{cases}
\]
Again, we will assume the left endpoint is uniformly distributed on the circumference of the circle.  We draw 100 i.i.d chords from $f(\theta)$ in the circle and count the number of regions. The process is repeated 1000 times and we plot the result into a histogram, which is shown in Figure 3.

\begin{figure}[h]
  \centering
  \includegraphics[scale=0.3]{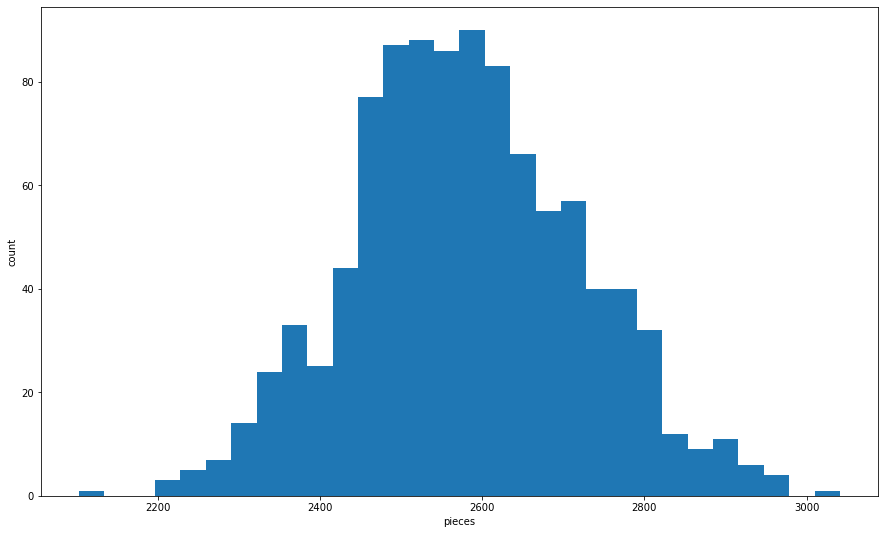}
  \caption{A histogram of 1000 simulations of the number of regions
created when 100 random chords are drawn according to the distance distribution $f(\theta) = \sin{\theta}$ and left endpoint uniformly distributed on the circumference of the circle.}
\end{figure}
The histogram appears to be a normal distribution. To test this hypothesis, we do a Kolmogorov–Smirnov test \cite{massey1951kolmogorov} on the distribution. To find the expected value and standard deviation of the regions, we need to borrow the results from equation (12) and (13) in the following sections. In this case, the expected value is 2576 and the standard deviation is 144.3. After normalizing the distribution and putting it into the (KS) test, we get statistic = 0.03 and p-value = 0.29, which is larger than 0.05. Therefore, the null hypothesis cannot be rejected. The cumulative distributions of the number of regions and normal distribution are shown in Figure 3.
\begin{figure}[h]
  \centering
  \includegraphics[scale=0.3]{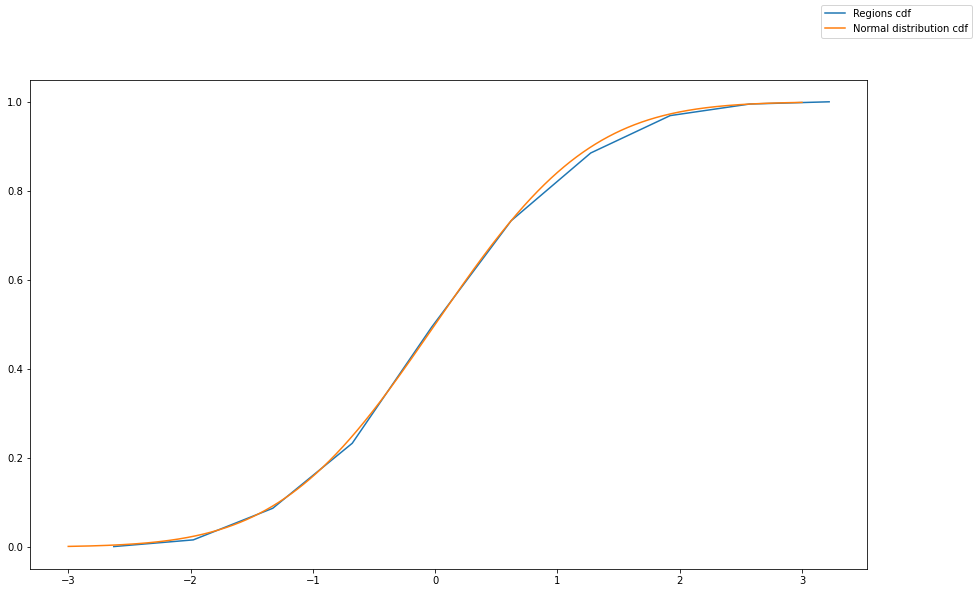}
  \caption{The comparison of the cumulative distribution between the regions and normal distribution.}
\end{figure}

To strengthen this hypothesis, we do simulation and test on various numbers of random chords, and the result is showing in the following table.

\begin{tabular}{lrrrr}
\toprule
{} &  number of chords &  repetition &  statistic &  p-value \\
\midrule
0 &              10.0 &       100.0 &       0.09 &     0.36 \\
1 &              20.0 &       100.0 &       0.12 &     0.08 \\
2 &              40.0 &       100.0 &       0.07 &     0.74 \\
3 &              80.0 &       100.0 &       0.07 &     0.73 \\
4 &             160.0 &       100.0 &       0.11 &     0.14 \\
\bottomrule
\end{tabular}\\

We can see all the simulations pass the test. Now we prove that the distribution converges to normal distribution using Stein's method for Central Limit Theorem. Here is the main result of this paper. It will be repeated and explained in the following sections.

For a function $f: \mathbb{R} \rightarrow \mathbb{R}$, denote 
\begin{align*}
    ||f|| = \sup_x |f(x)|.
\end{align*}
Suppose $h$ is a differential function that $||h||$ and $||h'||$ are finite. Then we denote $\Phi h$ to be $\mathbb{E}[h(Z)]$ and $\Phi(w)$ to be $P(Z<w)$ where $Z$ is a standard normal random variable. We also define notation $\Theta$ such that $a_n \sim \Theta(n^i)$ if there exist $k_2>k_1>0$ and $n_0$, and for all $n>n_0$, $k_1\cdot n^i \leq a_n \leq k_2\cdot n^i$.

\begin{theorem}
Suppose $F_n$ is the number of regions after drawing $n$ random chords. Then, its expectation, $\mathbb{E}[F_n]$ is $\sim \Theta(n^2)$ and its standard deviation, $\sigma_n \sim \Theta(n^{3/2})$. Moreover, for $n > 5$,
\begin{align*}
    \sup_{||h||\leq 1, ||h'||\leq 1}\left(
    \left|\mathbb{E}\left[h\left(\frac{F_n - \mathbb{E}(F_n)}{\sigma_n}\right) \right] - \Phi h \right|\right) \leq \frac{6n^{5/2}}{\sigma_n^2} + \frac{2n^4}{\sigma_n^3} \sim \Theta(n^{-1/2})
\end{align*}
and
\begin{align*}
    \sup_{w}\left(\left|P\left(\frac{F_n - \mathbb{E}[F_n]}{\sigma_n} < w \right) - \Phi(w)\right|\right) \leq 14\frac{n^4}{\sigma_n^3} \sim \Theta(n^{-1/2}).
\end{align*}
\end{theorem}
$\mathbb{E}(F_n)$ and $\sigma_n$ will be calculated in section 2, where Theorem 1.1 is proved in two parts.

\section{Proving Central Limit Theorem by Stein's method}
In the following proof, we will first express the number of regions in terms of the number of ``chord $i$ crosses chord $j$" events. These events satisfy an ``only locally dependent" property. We can then apply known theorems which implement the underlying Stein's method in order to give explicit general bounds for Normal approximation for such ``only locally dependent" events. The contribution of this paper is to evaluate these general bounds in the specific context of our random chords model.

Suppose we draw $n$ i.i.d. random chords (from $f(\theta)$ and left endpoint uniformly distributed on the circumference of the circle) in the circle. Let $R_n$ to be the total number of intersections after drawing $n$ i.i.d random chords. Let $A_{ij}$ be the indicator function that chords $i$ and $j$ cross each other
\[ A_{ij} =
  \begin{cases}
    1,       & \quad \text{if $i$ and $j$ intersect}\\
    0,  & \quad \text{else}.
  \end{cases}
\]
Then according to \cite{solomon1978geometric} page 134-137,
\begin{align}
    \mathbb{E}[R_n] = \sum_{i=1}^{n-1}\sum_{j=i+1}^{n} \mathbb{E}[A_{ij}] = {n \choose 2} \mathbb{E}[A_{12}],
\end{align}
and
\begin{align}
    \mathbb{E}[R_n^2] = {n \choose 2}\mathbb{E}[A_{12}]+6{n \choose 3}\mathbb{E}[A_{12}A_{13}] + 6{n \choose 4}(\mathbb{E}[A_{12}])^2.
\end{align}

As for $\mathbb{E}[A_{12}]$ and $\mathbb{E}[A_{12}A_{13}]$ in (1) and (2), we can bring $i=1$, $j=2$, $k=3$ into the following equations \cite{solomon1978geometric}:
\begin{align}
    \mathbb{E}[A_{ij}] = \frac{4}{\pi}\int_{0}^{\frac{\pi}{2}}\int_{0}^{\theta_i} \theta_j f(\theta_j)f(\theta_i)\;d\theta_j d\theta_i,
\end{align}
and
\begin{align}
    \mathbb{E}[A_{ij}A_{ik}] = \frac{4}{\pi^2} \int_{0}^{\frac{\pi}{2}}\left[\int_{0}^{\theta_i} \theta_j f(\theta_j) \;d\theta_j + \theta_i \int_{\theta_i}^{\frac{\pi}{2}} f(\theta_j)\; d\theta_j \right]^2 f(\theta_i)\; d\theta_i.
\end{align}

With a closer observation of (3) and (4), we find the Proposition 2.1, which will come useful in the following calculation. 
\begin{proposition}
$(\mathbb{E}[A_{ij}])^2 < E[A_{ij}A_{ik}]$ regardless of our choice of probability distribution function $f$ if $\theta$ is a continuous random variable. The inequality also implies $A_{ij}$ and $A_{ik}$ are dependent.
\end{proposition}
\begin{proof}
By symmetry,
\begin{align*}
    2\mathbb{E}[A_{ij}] &= \frac{4}{\pi}\int_{0}^{\frac{\pi}{2}}\int_{0}^{\theta_i} \theta_j  f(\theta_j)f(\theta_i)\;d\theta_j d\theta_i + \frac{4}{\pi}\int_{0}^{\frac{\pi}{2}}\int_{0}^{\theta_j} \theta_i f(\theta_i)f(\theta_j)\;d\theta_i d\theta_j\\
    &= \frac{4}{\pi}\int_{0}^{\frac{\pi}{2}}\left[\int_{0}^{\theta_i} \theta_j f(\theta_j) \;d\theta_j + \theta_i \int_{\theta_i}^{\frac{\pi}{2}} f(\theta_j)\; d\theta_j \right]f(\theta_i)\;d\theta_i.
\end{align*}
Suppose $\theta$ is a random variable with density $f$ and 
$g(\theta) = \int_{0}^{\theta} \theta_j f(\theta_j) \;d\theta_j + \theta \int_{\theta}^{\frac{\pi}{2}} f(\theta_j)\; d\theta_j$. Thus,
\begin{align*}
    (\mathbb{E}[A_{ij}])^2 = \frac{4}{\pi^2}(\mathbb{E}[g(\theta)])^2
\end{align*}
and
\begin{align*}
    \mathbb{E}[A_{ij}A_{ik}] = \frac{4}{\pi^2}\mathbb{E}[g(\theta)^2].
\end{align*}
Therefore, it is always the case that
\begin{align*}
    \frac{4}{\pi^2}\operatorname{Var}(g(\theta)) = \mathbb{E}[A_{ij}A_{ik}] - (\mathbb{E}[A_{ij}])^2 \geq 0
\end{align*}
and
\begin{align*}
    \mathbb{E}[A_{ij}A_{ik}] - (\mathbb{E}[A_{ij}])^2 = 0
\end{align*}
if and only if $g(\theta)$ is a constant. Suppose $F(\theta)$ is the cumulative distribution of $\theta$ and $F(0) = 0$, $F(\frac{\pi}{2}) = 1$. Since $f(\theta)$ is the density function, by integrating by parts,, we get
\begin{align*}
    g(\theta) &=  \theta_jF(\theta_j)\Big|_{0}^{\theta} - \int_{0}^{\theta} F(\theta_j)\;d\theta_j + \theta(1-F(\theta))\\
    &= \theta - \int_{0}^{\theta}F(\theta_j)\;d\theta_j
\end{align*}
Suppose $g(\theta)$ is a non-random constant $c$. Then, $\int_{0}^{\theta}1-F(\theta_j)\;d\theta_j = c$. The only solution for $F(\theta_j)$ would be $F(\theta_j)=1$, Lebesgue almost everywhere on $[0, \pi/2]$. Since $F$ is continuous as it has a density $f$, the previous statement implies that $F(\theta_j)=1$ everywhere. However, this contradicts the fact that $F(0)$ must be zero. Thus a contradiction arises and $g(\theta)$ cannot be a constant random variable if $f(\theta)$ is a continuous function. As a result, $(\mathbb{E}[A_{ij}])^2 < \mathbb{E}[A_{ij}A_{ik}]$ is $f(\theta)$ is a density distribution. Since $\operatorname{Cov}(A_{ij}A_{ik}) = \mathbb{E}[A_{ij}A_{ik}] - (\mathbb{E}[A_{ij}])^2 > 0$, $A_{ij}$ and $A_{ik}$ are dependent.
\end{proof}

When we draw $n$ chords on the circle we can create
a graph whose vertices are any points of intersections between two chords
or an intersection of the chord and the circle. Then Euler’s characteristic \cite{euler1758elementa} for planar graphs tells us that
\begin{align*}
    V_n - E_n + F_n = 2,
\end{align*}
where after drawing $n$ chords, $V_n$ means the number of vertices; $E_n$ means the number of edges; $F_n$ means the number of regions. Since we only care about the regions inside the circle, we need to exclude the one region that is outside the circle. Therefore, we redefine $F_n$ to be the number of regions inside the circle and an updated formula would be 
\begin{align}
    V_n - E_n + F_n = 1.
\end{align}

For example, in Figure 2, $V_n = 11$, $E_n = 18$, $F_n = 8$. Now our target is to find the number of regions. Suppose $n$ i.i.d. random chords are drawn and $R_n$ intersections are created by chords only. Each vertex of the above graph corresponds to the intersection
of two chords, or to an intersection of a chord with the circle. By assumption there are $R_n$ vertices of the first type. Since each chord intersects the circle twice, and the probability for two random chords to intersect the circle at the same point is zero as the left endpoint has a continuous distribution on the circumference of the circle, there are 2n vertices of the latter type. Therefore,
\begin{align}
    V_n = R_n + 2n.
\end{align}

As for edges, the degree of the vertices created by the intersection of two chords is 4. The degree of the vertices created by the intersection one chord and the circle is 3. Every edge is adjacent to two vertices. Therefore,
\begin{align}
    E_n = \frac{4\cdot R_n + 3\cdot 2n}{2} = 2R_n + 3n.
\end{align}

Thus, bringing (6) and (7) to (5), we get
\begin{align*}
    F_n = R_n + n + 1.
\end{align*}

Recall by definition of $R_n$ and $A_{ij}$, $R_n = \sum_{i=1}^{n-1}\sum_{j=i+1}^{n} A_{ij}$. Thus, after expanding $R_n$, the number of regions can be written as
\begin{align}
    F_n = \sum_{i=1}^{n-1}\sum_{j=i+1}^{n} A_{ij} + n + 1.
\end{align}

With an easy observation of number of regions ($F_n$), we find that it is the summation of dependent Bernoulli random variables ($A_{ij}$). We first calculate its expected value and standard deviation. From (8), we get
\begin{align}
    \mathbb{E}[F_n] = \frac{1}{2}n(n-1)\mathbb{E}[A_{12}]+n+1 \sim \Theta(n^2).
\end{align}
From (1), (2) and Proposition 2.1, we get
\begin{align}
    \operatorname{Var}(R_n) &= \mathbb{E}[R_n^2]-(\mathbb{E}[R_n])^2 \nonumber\\
    &= \frac{1}{4}\left[2n(n-1)\mathbb{E}[A_{12}]+4n(n-1)(n-2)\mathbb{E}[A_{12}A_{13}] - n(n-1)(4n-6)(\mathbb{E}[A_{12}])^2\right] \sim \Theta(n^3).
\end{align}
Thus,
\begin{align}
    \sigma_n &= \sqrt{\operatorname{Var}(F_n)} = \sqrt{\operatorname{Var}(F_n-n-1)} = \sqrt{\operatorname{Var}(R_n)}\nonumber\\ 
    &= \frac{1}{2}\left[2n(n-1)\mathbb{E}[A_{12}]+4n(n-1)(n-2)\mathbb{E}[A_{12}A_{13}] - n(n-1)(4n-6)(\mathbb{E}[A_{12}])^2\right]^{\frac{1}{2}} \sim \Theta(n^{3/2}).
\end{align}
Before we get any deeper, recall in the simulation part we calculate the expectation and standard deviation of $F_n$ when $f(\theta) = \sin(\theta)$ and $n=100$. Let's do the calculation here. According to (3) and (4), $\mathbb{E}[A_{12}] = \frac{1}{2}$ and $\mathbb{E}[A_{12}A_{13}] = \frac{8}{3\pi^2}$. Then, we bring them into (9) and (11) and get
\begin{align}
    E[F_n] = \frac{100\cdot99}{2}\cdot\frac{1}{2} + 100 + 1 = 2576
\end{align}
\begin{align}
    \sigma_n = \frac{1}{2}\left[2\cdot100\cdot99\cdot\frac{1}{2} + 4\cdot100\cdot99\cdot98\cdot\frac{8}{3\pi^2} - 100\cdot99\cdot394\cdot\frac{1}{4} \right]^\frac{1}{2} = 144.3
\end{align}

With the expected value and standard deviation in (9) and (11), we can shift and scale the distribution to make it have a mean of 0 and a variance of 1 as mentioned in the introduction. Therefore, we would like to prove after shifting and scaling, the distribution would converges to a standard normal distribution.

To prove $F_n$ converges to normal distribution, we would want to show for all function $h$ such that $||h||$ and $||h'||$ are finite,
\begin{align*}
    \left|\mathbb{E}\left[h\left(\frac{F_n - \mathbb{E}(F_n)}{\sigma_n}\right) \right] - \Phi h \right|
\end{align*}
is small when $n$ is large. Stein's method is first developed by Stein in his paper \cite{stein1972bound} and explained in detail in \cite{ross2011fundamentals}. We will apply the method to our problem by borrowing the following Lemma 2.2 from \cite{rinott2000normal} on page 21 Theorem 2.1.
\begin{lemma}
Suppose $Y_1,Y_2,...,Y_n$ are random variables whose expected value is 0 and $\operatorname{Var}[\sum_{i=1}^{n} Y_i] = \sigma$. Let $M_i$ denote the dependent neighborhood of $Y_i$, which means $Y_i$ and $Y_j$ are dependent if and only if $j\in M_i$. Then
\begin{align}
    \left|\mathbb{E}\left[h\left(\frac{\sum_{i=1}^{n} Y_i}{\sigma} \right)\right] - \Phi h\right| \leq \frac{2}{\sigma^2}||h||\sqrt{\mathbb{E}\left[\left(\sum_{i=1}^{n} \sum_{j\in M_i} (Y_iY_j - \mathbb{E}[Y_iY_j])\right)^2\right]} + \frac{1}{\sigma^3}||h'||\mathbb{E}\left[\sum_{i=1}^{n}|Y_i|(\sum_{j\in M_i}Y_j)^2 \right].
\end{align}
\end{lemma}
Now, we can prove the following theorem for our problem.

\begin{theorem}
For $n >5$,
\begin{align*}
    \left|\mathbb{E}\left[h\left(\frac{F_n - \mathbb{E}(F_n)}{\sigma_n}\right) \right] - \Phi h \right| \leq \frac{6n^{5/2}}{\sigma_n^2}||h|| + \frac{2n^4}{\sigma_n^3}||h'|| \sim \Theta(n^{-1/2})
\end{align*}
\end{theorem}
\begin{proof}

To apply (14) on our problem, we first need to reduce the expected value of random variables to 0. Suppose,
\begin{align*}
    Y_{i,j} = A_{ij}-\mathbb{E}[A_{ij}].
\end{align*}
We denote $M_{\{i,j\}}$ to be the dependent neighborhood for $Y_{i,j}$. According to \cite{solomon1978geometric} page 135, $A_{ij}$ and $A_{kl}$ are independent if $\{i,j\}\cap{\{k,l\}}=\emptyset$. And according to Proposition 2.1, $A_{ij}$ and $A_{kl}$ are dependent if $\{i,j\}\cap{\{k,l\}}\neq \emptyset$.
We have
\begin{align}
    {\{k,l\}} \notin M_{\{i,j\}} \text{ if and only if $\{i,j\}\cap{\{k,l\}}=\emptyset$}.
\end{align}

According to (8) and (14), we get
\begin{align}
    \left|\mathbb{E}\left[h\left(\frac{F_n - \mathbb{E}(F_n)}{\sigma_n} \right)\right] - \Phi h\right| &= \left|\mathbb{E}\left[h\left(\frac{\sum_{i=1}^{n-1}\sum_{j=i+1}^{n} Y_{i,j}}{\sigma_n} \right)\right] - \Phi h\right|\nonumber\\
    &\leq \frac{2}{\sigma_n^2}||h||\sqrt{\mathbb{E}\left[\left(\sum_{i=1}^{{n-1}}\sum_{j=i+1}^{n}\sum_{{\{k,l\}} \in M_{\{i,j\}}} (Y_{i,j}Y_{k,l}-\mathbb{E}[Y_{i,j}Y_{k,l}])\right)^2\right]}\nonumber \\
    &+ \frac{1}{\sigma_n^3}||h'||\mathbb{E}\left[\sum_{i=1}^{{n-1}}\sum_{j=i+1}^{n} |Y_{i,j}|\left(\sum_{{\{k,l\}} \in M_{\{i,j\}}} Y_{k,l}\right)^2 \right].
\end{align}

Therefore, our next step is to bound the right hand side of (16). We start with the first part,

\begin{align*}
&\mathbb{E}\left[\left(\sum_{i=1}^{{n-1}}\sum_{j=i+1}^{n}\sum_{{\{k,l\}} \in M_{\{i,j\}}} Y_{i,j}Y_{k,l}-\mathbb{E}[Y_{i,j}Y_{k,l}]\right)^2\right] \nonumber\\
&=\mathbb{E}\left[\sum_{i=1}^{{n-1}}\sum_{j=i+1}^{n}\sum_{{\{k,l\}} \in M_{\{i,j\}}} \sum_{i'=1}^{{n-1}}\sum_{j'=i+1}^{n}\sum_{\{k',l'\} \in M_{\{i',j'\}}} (Y_{i,j}Y_{k,l}-\mathbb{E}[Y_{i,j}Y_{k,l}])(Y_{i',j'}Y_{k',l'}-\mathbb{E}[Y_{i',j'}Y_{k',l'}])\right]
\end{align*}

The number of terms in $\sum_{i'=1}^{{n-1}}\sum_{j'=i+1}^{n}\sum_{\{k',l'\} \in M_{\{i',j'\}}}$ that $|\{i',j'\}\cap\{k',l'\}|=1$ is $\frac{n(n-1)}{2}\cdot 2(n-2)$. And the number of elements that $|\{i',j'\}\cap\{k',l'\}|=2$ is $\frac{n(n-1)}{2}$. When $|\{i',j'\}\cap\{k',l'\}|=0$, $[k',l'] \notin M_{\{i',j'\}}$ and we do not need to consider such case because of (15). By symmetry we can assume $Y_{i',j'}Y_{k',l'}$ to be $Y_{1,2}Y_{1,3}$ when $|\{i',j'\}\cap\{k',l'\}|=1$ and $Y_{i',j'}Y_{k',l'}$ to be $Y_{1,2}^2$ when $|\{i',j'\}\cap\{k',l'\}|=2$. Thus,
\begin{align*}
&\mathbb{E}\left[\left(\sum_{i=1}^{{n-1}}\sum_{j=i+1}^{n}\sum_{{\{k,l\}} \in M_{\{i,j\}}} Y_{i,j}Y_{k,l}-\mathbb{E}[Y_{i,j}Y_{k,l}]\right)^2\right] \nonumber\\
&=\frac{n(n-1)}{2}\cdot 2(n-2) \mathbb{E}\left[\sum_{i=1}^{{n-1}}\sum_{j=i+1}^{n}\sum_{{\{k,l\}} \in M_{\{i,j\}}} (Y_{i,j}Y_{k,l}-\mathbb{E}[Y_{i,j}Y_{k,l}])(Y_{1,2}Y_{1,3}-\mathbb{E}[Y_{1,2}Y_{1,3}])\right]\nonumber\\
&+ \frac{n(n-1)}{2}\cdot \mathbb{E}\left[\sum_{i=1}^{{n-1}}\sum_{j=i+1}^{n}\sum_{{\{k,l\}} \in M_{\{i,j\}}} (Y_{i,j}Y_{k,l}-\mathbb{E}[Y_{i,j}Y_{k,l}])(Y_{1,2}^2-\mathbb{E}[Y_{1,2}^2])\right].
\end{align*}

Note that $Y_{i,j}Y_{k,l}-\mathbb{E}[Y_{i,j}Y_{k,l}]$ and $Y_{i',j'}Y_{k',l'}-\mathbb{E}[Y_{i',j'}Y_{k',l'}]$ are independent if $\{i,j,k,l\}\cap\{i',j',k',l'\} = \emptyset$ as every chord is independent of each other. In such case, we can break the multiplication inside $\mathbb{E}[\cdot]$ apart and get
\begin{align*}
    \mathbb{E}[(Y_{i,j}Y_{k,l}-\mathbb{E}[Y_{i,j}Y_{k,l}])(Y_{1,2}Y_{1,3}-\mathbb{E}[Y_{1,2}Y_{1,3}])] 
    &= \mathbb{E}[Y_{i,j}Y_{k,l}-\mathbb{E}[Y_{i,j}Y_{k,l}]] \cdot \mathbb{E}[Y_{1,2}Y_{1,3}-\mathbb{E}[Y_{1,2}Y_{1,3}]] \nonumber\\
    &= 0 \text{ if } \{i,j,k,l\}\cap\{1,2,3\}=\emptyset.
\end{align*}

\begin{align*}
    \mathbb{E}[(Y_{i,j}Y_{k,l}-\mathbb{E}[Y_{i,j}Y_{k,l}])(Y_{1,2}^2-\mathbb{E}[Y_{1,2}^2])] 
    &= \mathbb{E}[Y_{i,j}Y_{k,l}-\mathbb{E}[Y_{i,j}Y_{k,l}]] \cdot \mathbb{E}[Y_{1,2}^2-\mathbb{E}[Y_{1,2}^2]] \nonumber\\
    &= 0 \text{ if } \{i,j,k,l\}\cap\{1,2\}=\emptyset.
\end{align*}

Thus, let $I(\cdot)$ be the indicator function of event $\cdot$,
\begin{align*}
    &\mathbb{E}\left[\left(\sum_{i=1}^{{n-1}}\sum_{j=i+1}^{n}\sum_{{\{k,l\}} \in M_{\{i,j\}}} Y_{i,j}Y_{k,l}-\mathbb{E}[Y_{i,j}Y_{k,l}]\right)^2\right] \nonumber\\
    &\leq \frac{n(n-1)}{2}\cdot 2(n-2) \mathbb{E}\left[\sum_{i=1}^{{n-1}}\sum_{j=i+1}^{n}\sum_{{\{k,l\}} \in M_{\{i,j\}}} I(\{i,j,k,l\} \cap \{1,2,3\} \neq \emptyset) \right]\nonumber\\
    &+ \frac{n(n-1)}{2}\cdot \mathbb{E}\left[\sum_{i=1}^{{n-1}}\sum_{j=i+1}^{n}\sum_{{\{k,l\}} \in M_{\{i,j\}}}I(\{i,j,k,l\} \cap \{1,2\} \neq \emptyset)\right]\nonumber\\
    &= n(n-1)(n-2)\cdot (9n^2-24n) + \frac{n(n-1)}{2}\cdot (6n^2-10n-3)\nonumber\\
    &= 9n^5-48n^4+82n^3-\frac{89n^2}{2}+\frac{3n}{2}.
\end{align*}
Since
\begin{align*}
    -48n^4+82n^3-\frac{89n^2}{2}+\frac{3n}{2}&= -(48n^4-82n^3+15n^2) - (\frac{59n^2}{2}-\frac{3n}{2})\\
    &= -n^2(2n - 3)(24n - 5) - \frac{n}{2}(59n - 3)\\
    &< 0 \text{ if } n>5,
\end{align*}
we have
\begin{align*}
    &\mathbb{E}\left[\left(\sum_{i=1}^{{n-1}}\sum_{j=i+1}^{n}\sum_{{\{k,l\}} \in M_{\{i,j\}}} Y_{i,j}Y_{k,l}-\mathbb{E}[Y_{i,j}Y_{k,l}]\right)^2\right] \leq 9n^5 \text{ for all $n>5$}.
\end{align*}

Bring it back to (16). For all $n>5$,
\begin{align*}
    \frac{2}{\sigma_n^2}||h||\sqrt{\mathbb{E}\left[\left(\sum_{i=1}^{{n-1}}\sum_{j=i+1}^{n}\sum_{{\{k,l\}} \in M_{\{i,j\}}} (Y_{i,j}Y_{k,l}-\mathbb{E}[Y_{i,j}Y_{k,l}])\right)^2\right]}
    &\leq \frac{2}{\sigma_n^2}||h||\sqrt{9n^5}\nonumber\\
    &= \frac{6n^{5/2}}{\sigma_n^2}||h||.
\end{align*}

As for the second part of (16),
\begin{align*}
    \frac{1}{\sigma_n^3}||h'||\mathbb{E}\left\{\sum_{i=1}^{{n-1}}\sum_{j=i+1}^{n} |Y_{i,j}|\left(\sum_{{\{k,l\}} \in M_{\{i,j\}}} Y_{k,l}\right)^2 \right\}
    &\leq \frac{1}{\sigma_n^3}||h'||\mathbb{E}\left\{\sum_{i=1}^{{n-1}}\sum_{j=i+1}^{n} 1\left(\sum_{{\{k,l\}} \in M_{\{i,j\}}} 1\right)^2 \right\}\nonumber\\
    &= \frac{1}{\sigma_n^3}||h'||\left(\frac{n(n-1)}{2}(2n-3)^2\right)\nonumber\\
    &= \frac{1}{\sigma_n^3}||h'||\left(2n^4-8n^3+\frac{21n^2}{2}-\frac{9n}{2}\right).
\end{align*}

Since $-8n^3+\frac{21n^2}{2}-\frac{9n}{2} = -\frac{n}{2}(16n - 5)(n - 1) - 2n < 0$ for all $n>5$,

\begin{align*}
    \frac{1}{\sigma_n^3}||h'||\mathbb{E}\left\{\sum_{i=1}^{{n-1}}\sum_{j=i+1}^{n} |Y_{i,j}|\left(\sum_{{\{k,l\}} \in M_{\{i,j\}}} Y_{k,l}\right)^2 \right\} \leq \frac{2n^4}{\sigma_n^3}||h'||.
\end{align*}

Since $\sigma_n \sim \Theta(n^{3/2})$, from (16) we get
\begin{align*}
    \left|\mathbb{E}\left[h\left(\frac{F_n - \mathbb{E}(F_n)}{\sigma_n}\right) \right] - \Phi h \right| \leq \frac{6n^{5/2}}{\sigma_n^2}||h|| + \frac{2n^4}{\sigma_n^3}||h'|| \sim \Theta(n^{-1/2}).
\end{align*}
\end{proof}

Another way to bound the difference between the distribution of $F_n$ and the normal distribution is to measure the difference of c.d.f between them. In other word,
\begin{align}
    \left|P\left(\frac{F_n - \mathbb{E}[F_n]}{\sigma_n} < w \right) - \Phi(w)\right|
\end{align}
needs to be small when $n$ is large. To bound this term, we need to use the following Lemma 2.4, which restates Theorem 2.2 from \cite{rinott2000normal}. The detailed proof can be found in \cite{dembo1996some}.

\begin{lemma}
Let $Y_1,... ,Y_n$ be random variables satisfying $|Y_i - \mathbb{E}[Y_i]| \leq B$ a.s., $i = 1, . . . , n$ $\mathbb{E}[\sum_{i=1}^{n} Y_i] = \lambda$, $\operatorname{Var}[\sum_{i=1}^{n} Y_i] = \sigma^2 > 0$ and $\frac{1}{n}\mathbb{E}\left[\sum_{i=1}^{n} |Y_i - \mathbb{E}[Y_i]|\right] = \mu $.  Let $M_i \subset \{1,... ,n\}$ be such that $j \in M_i$ if and only if $i \in M_j$ and $(Y_i, Y_j )$ is independent of $\{Y_k\}_{k \notin M_i \cup M_j}$ for $i, j =
1,... ,n$ , and set $D = max_{1\leq i \leq n} |M_i|$. Then

\begin{align*}
    \left|P\left(\frac{\sum_{i=1}^{n}Y_i - \lambda}{\sigma} \leq w \right) - \Phi(w) \right| \leq 7\frac{n\mu}{\sigma^3}(DB)^2.
\end{align*}
\end{lemma}

The definition of $M_i$ might be ambiguous in lemma 2.4, but it basically means the set of random variables in $Y_1,...,Y_n$ that is dependent of $Y_i$. With the help of this inequality, we can bound (17) in the following theorem.

\begin{theorem}
\begin{align*}
    \left|P\left(\frac{F_n - \mathbb{E}[F_n]}{\sigma_n} < w \right) - \Phi(w)\right| \leq 14\frac{n^4}{\sigma_n^3} \sim \Theta(n^{-1/2})
\end{align*}
\end{theorem}
\begin{proof}
For our problem, after eliminating the constant, the target random variable is $\sum_{i=1}^{{n-1}}\sum_{j=i+1}^{n} A_{ij}$ according to (8). $\mathbb{E}[F_n]$ can be found in (9). The standard deviation $\sigma = \sigma_n$ according to (11). The total number of random variables is ${n\choose 2} = \frac{n(n-1)}{2}$. Since $A_{ij}$ is Bernoulli random variable, $|A_{ij}-\mathbb{E}[A_{ij}]| \leq 1$. Therefore, $B\leq 1$ and
\begin{align*}
    \mu &= \frac{2}{n(n-1)}\mathbb{E}\left[\sum_{i=1}^{{n-1}}\sum_{j=i+1}^{n} |A_{ij} - \mathbb{E}[A_{ij}]|\right]\\
    &\leq  \frac{2}{n(n-1)}\mathbb{E}\left[\sum_{i=1}^{{n-1}}\sum_{j=i+1}^{n} 1\right] = 1.
\end{align*}

$A_{ij}$ is dependent on $A_{kl}$ if and only if $|\{i,j\}\cap{\{k,l\}}|=1$ or $2$,  so $D = 2n-3$. Finally, we can bring everything into Lemma 2.4.
\begin{align*}
    \left|P\left(\frac{F_n - \mathbb{E}[F_n]}{\sigma_n} < \epsilon \right) - \Phi(\epsilon)\right| 
    &\leq 7\frac{\frac{n(n-1)}{2} \cdot 1}{\sigma_n^3} ((2n-3)\cdot1)^2\\
    &\leq 7\frac{n^2}{2\sigma_n^3}(2n)^2\\
    &= 14\frac{n^4}{\sigma_n^3}.
\end{align*}
Since $\sigma_n \sim \Theta(n^{3/2})$,
\begin{align*}
    \left|P\left(\frac{F_n - \mathbb{E}[F_n]}{\sigma_n} < \epsilon \right) - \Phi(\epsilon)\right| \leq 14\frac{n^4}{\sigma_n^3} \sim \Theta(n^{-1/2}).
\end{align*}
\end{proof}

Note that the random chords for our problem have a uniformly distributed left endpoint on the circumference of the circle. The reason for this condition is that there is no convenient formula to calculate $E[A_{ij}]$ and $E[A_{ij}A_{ik}]$ if the distribution of left endpoint is nonuniform or dependent of distance distribution $f(\theta)$. Thus, we cannot calculate the expectation and the variance of $F_n$. However, we can still prove the following corollary using the same reasoning.
\begin{theorem}
Suppose there are $n$ i.i.d. random chords in the circle. For every chord, the left endpoint may not uniformly distributed on the circumference of the circle and may be dependent on the half angle distribution $f$. However, each chord is constructed in a way such that the probability for three or more points to create an intersection at the same point or two or more points create an intersection with the circle at the same point is zero. Suppose $N = 7\frac{n\mu}{\sigma_n^2}(DB)^2$ as in the RHS of the formula in Lemma 2.4. If $0<(\mathbb{E}[A_{12}])^2<\mathbb{E}[A_{12}A_{13}]<1$, 
\begin{align*}
    \left|P\left(\frac{F_n - \mathbb{E}[F_n]}{\sigma_n} < w \right) - \Phi(w)\right| \leq N \sim \Theta(n^{-1/2}).
\end{align*}
If $0<(\mathbb{E}[A_{12}])^2=\mathbb{E}[A_{12}A_{13}]<1$, 
\begin{align*}
    \left|P\left(\frac{F_n - \mathbb{E}[F_n]}{\sigma_n} < w \right) - \Phi(w)\right| \leq N \sim \Theta(n^{-1}).
\end{align*}
\end{theorem}

\begin{proof}
Since there is no overlapping point, the reasoning about Euler's characteristic still holds and
\begin{align*}
    F_n = \sum_{i=1}^{n-1}\sum_{j=i+1}^{n} A_{ij} + n + 1.
\end{align*}
$F_n$ is still the sum of Bernoulli random variables, so out next step is to bring everything into Lemma 2.4. We first consider the case where  $0<(\mathbb{E}[A_{12}])^2<\mathbb{E}[A_{12}A_{13}]<1$. According to (11), $\sigma_n \sim \Theta(n^{3/2})$. By only considering the order, we have $\mu \sim \Theta(1)$, $B \sim \Theta(1)$. Since $(\mathbb{E}[A_{12}])^2<\mathbb{E}[A_{12}A_{13}]$, $A_{ij}$ is dependent of $A_{ik}$ and $D\sim \Theta(n^1)$. Finally, the total number of random variables is $\frac{n(n-1)}{2}\sim \Theta(n^2)$. Thus,
\begin{align*}
    \left|P\left(\frac{F_n - \mathbb{E}[F_n]}{\sigma_n} < w \right) - \Phi(w)\right| \leq N \sim \frac{\Theta(n^2)\cdot \Theta(n^1)^2}{\Theta(n^{3/2})^3} \sim \Theta(n^{-1/2}).
\end{align*}
When $0<(\mathbb{E}[A_{12}])^2=\mathbb{E}[A_{12}A_{13}]<1$, according to (11), $\sigma_n \sim \Theta(n^1)$. Again, $\mu \sim \Theta(1)$, $B \sim \Theta(1)$, and the total number of random variables is $\frac{n(n-1)}{2}\sim \Theta(n^2)$. Since $(\mathbb{E}[A_{12}])^2=\mathbb{E}[A_{12}A_{13}]$, $A_{ij}$ is independent of $A_{ik}$ and $D\sim \Theta(1)$. Thus,
\begin{align*}
    \left|P\left(\frac{F_n - \mathbb{E}[F_n]}{\sigma_n} < w \right) - \Phi(w)\right| \leq N \sim \frac{\Theta(n^2)\cdot \Theta(1)^2}{\Theta(n^{1})^3} \sim \Theta(n^{-1}).
\end{align*}
\end{proof}

A natural question one would ask from Theorem 2.6 is what will happen if $(\mathbb{E}[A_{12}])^2>\mathbb{E}[A_{12}A_{13}]$ as Proposition 2.1 is no longer useful. The case is that when we bring such situation into (10), we find $\operatorname{Var}(R_n) < 0$ as $n$ goes to infinity. Therefore, $(\mathbb{E}[A_{12}])^2\leq \mathbb{E}[A_{12}A_{13}]$ at all times.

\section{Acknowledgments}
This paper would not be possible without the help from my mentor Professor Soumik Pal, who has provided the main idea of Stein's method for the problem and commented on every sentence of this paper. I am grateful to Professor David Aldous for suggesting references and giving comments for this paper. I am also thankful to my friend Ami Oka for providing help on the computer simulation.

\bibliographystyle{apacite}
\bibliography{References}

\end{document}